\documentclass[10pt,a4paper]{amsart}
\usepackage[utf8]{inputenc}
\usepackage{amsmath}
\usepackage{amsfonts}
\usepackage{amssymb}
\usepackage{amsthm}
\usepackage{enumerate}
\usepackage{comment}
\usepackage{graphicx}
\usepackage{epstopdf}
\usepackage{hyperref}

\newcommand{\ol}[1]{\mkern 1.7mu\overline{\mkern-1.7mu#1\mkern-1.7mu}\mkern 1.7mu}

\def\R{\mathbb{R}}
\def\N{\mathbb{N}}

\def\O{\Omega}
\def\P{\mathcal{P}}
\def\p{\partial}

\def\theta{{\vartheta}}
\def\phi{{\varphi}}
\def\eps{{\varepsilon}}

\DeclareMathOperator{\dist}{dist}
\DeclareMathOperator{\divergence}{div}
\DeclareMathOperator{\graph}{graph}

\newtheorem{theorem}{Theorem}[section]
\newtheorem{lemma}[theorem]{Lemma}
\newtheorem{proposition}[theorem]{Proposition}

\theoremstyle{definition}
\newtheorem{definition}[theorem]{Definition}
\newtheorem{remark}[theorem]{Remark}
\newtheorem{example}[theorem]{Example}

\numberwithin{equation}{section}

\begin{document}

\title[Shadows of graphical mean curvature flow]{Shadows of graphical mean curvature flow}

% Information for first author
\author{Wolfgang Maurer}

\address{Wolfgang Maurer, Fachbereich Mathematik und Statistik, Universit\"at Konstanz, 78457 Konstanz, Germany}

%\email{wolfgang.maurer@uni-konstanz.de}

%\urladdr{www.math.uni-konstanz.de/~maurer/}

%\thanks{}

%\date{\today.}

%\dedicatory{}

%\keywords{}

\subjclass[2010]{53C44}

\begin{abstract}
We consider mean curvature flow of an initial surface that is the graph of a function over some domain of definition in $\R^n$.
If the graph is not complete then we impose a constant Dirichlet boundary condition at the boundary of the surface.
We establish longtime-existence of the flow and investigate the projection of the flowing surface onto $\R^n$, the shadow of the flow.
This moving shadow can be seen as a weak solution for mean curvature flow of hypersurfaces in $\R^n$ with a Dirichlet boundary condition.

Furthermore, we provide a lemma of independent interest to locally mollify the boundary of an intersection of two smooth open sets in a way that respects curvature conditions.
\end{abstract}

\maketitle

%\tableofcontents

%-----------------------   mainmatter   ------------------------------------------

\allowdisplaybreaks

\section{Introduction}

A family $(M_t)_{t\in (0,T)}$ of hypersurfaces of $\R^{n+1}$ is said to move by mean curvature flow if there is a map $X\colon M\times (0,T)\to \R^{n+1}$ such that $X(\cdot,t)$ is an immersion for all $t\in (0,T)$ with $X(M,t)=M_t$ and $X$ solves
\[
	\frac{d}{dt}X(p,t)= -H(p,t)\, \nu(p,t),
\]
where $M$ is a $n$-dimensional manifold, $H(\cdot,t)$ is the mean curvature of $M_t$ and $\nu(\cdot,t)$ its normal, such that $-H\nu$ is the mean curvature vector.
If $M_t=\graph u(\cdot,t)$ for a family of functions $u(\cdot,t)\colon \O\subset\R^n \to \R$, then $M_t$ moves by mean curvature flow if and only if $u$ solves the graphical mean curvature flow equation, which is the parabolic partial differential equation
\begin{equation} \label{eq GMCF} \tag{GMCF}
  \frac{d}{dt}u = \sqrt{1+|Du|^2}\,\divergence\left(\frac{Du}{\sqrt{1+|Du|^2}}\right).
\end{equation}

Graphical mean curvature flow was studied in \cite{E+H} by Ecker and Huisken. They proved long-time existence for the mean curvature flow of entire graphs and showed that the solution stays graphical for all time.
More recently, S\'{a}ez Trumper and Schn\"urer proved in \cite{MCF without Singularities} a long-time existence result for complete graphs. Starting from an open set $\O_0$ and a proper function $u_0\colon \O\to \R_+$, they showed the existence of a solution $u$ to graphical mean curvature flow with initial data $u_0$, where $u(\cdot,t)$ is defined on an open set $\O_t$ for $t\geq 0$. This solution will not develop singularities on a finite level but it can disappear to infinity forming a singularity at infinity-level. It was observed that the sets $\p \O_t$ can be interpreted as a weak solution to mean curvature flow, starting from $\p \O_0$, and that it coincides almost everywhere with the level-set flow as long as the latter does not fatten (Fig.\,\ref{GraphBall}).
\begin{figure}
  \centering
  \setlength{\unitlength}{1cm}
  \begin{picture}(10,5)
    \put(2,0){\includegraphics[height=5cm]{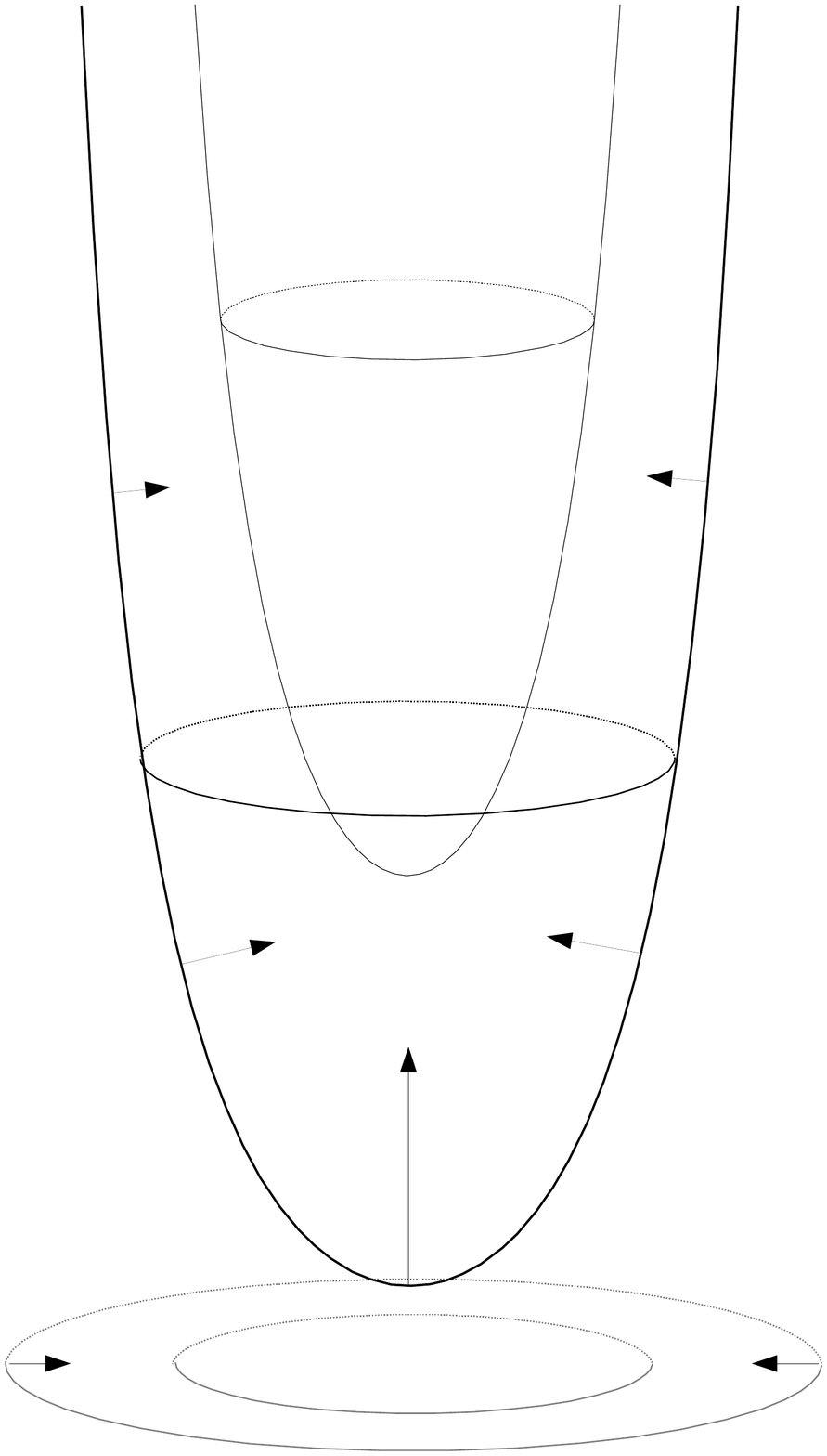}}
    \put(5.5,4){graph}
    \put(5.6,0.4){shadow}
  \end{picture}
  \caption{A rotationally symmetric ``shadowflow'' where we consider a graph over a ball. The graph moves by mean curvature flow and the shadow evolves by mean curvature flow too, but in a weak sense. In this setting the graph will disappear to infinity in finite time, while the shadow develops a point singularity. Generally, any singularity occuring on the shadow-level will happen at infinity on the graph-level.}
  \label{GraphBall}
\end{figure}
In \cite{Q_k} and \cite{K^alpha} existence results analogous to that of S\'{a}ez and Schn\"urer are proven for some fully nonlinear flows of non-compact convex surfaces.

In this article we consider graphical mean curvature flow with a Dirichlet boundary condition.
The existence result of this article reads as follows.
\begin{theorem} \label{thm existence}
 Let $\O\subset \R^n$ be open, smooth, and mean convex, that is, its boundary $\p \O\in C^\infty$ has nonnegative mean curvature $H[\p \O]\geq 0$.
 Let $u_0\colon \ol{\O} \to [-\infty,\infty]$ be continuous and assume $u_0$ is locally Lipschitz in the set $\{x\in \ol{\O}\colon |u_0(x)|<\infty\}$ and $u_0|_{\p \O}$ is of class $C^2$ in $\{x\in \p\O \colon |u_0(x)|<\infty\}$.
 Then there is a continuous function $u\colon \ol{\O}\times [0,\infty) \to [-\infty,\infty]$ which is smooth on $\big(\O\times (0,\infty)\big)\cap\{|u|<\infty\}$ and solves
 \[
  \begin{cases}
   \dot{u}= \sqrt{1+|Du|^2}\, \divergence\left(\frac{Du}{\sqrt{1+|Du|^2}}\right)
   & \text{in } \big(\O\times (0,\infty)\big)\cap\{|u|<\infty\},\\
   u(x,t)= u_0(x)
   & \text{for } (x,t)\in \P(\O\times (0,\infty)),
  \end{cases}
 \]
 where $\P(\O\times (0,\infty)):= (\O\times \{0\})\cup (\p \O\times [0,\infty))$.
\end{theorem}
The condition of nonnegative mean curvature for the boundary $\p \O$ is necessary to expect longtime existence for the graphical mean curvature flow when considering the Dirichlet problem for general boundary data. Otherwise, the gradient of $u$ may become unbounded in finite time and the flowing surface may cease to be graphical.

Note that for $\O= \R^n$ Theorem \ref{thm existence} is a generalization of the result in \cite{MCF without Singularities} because we do not need the assumptions $u_0$ proper and $u_0\geq 0$.

The proof of Theorem \ref{thm existence} is the subject of Section \ref{sec existence}. It involves an approximation of the problem by bounded auxiliary problems and uses an Arzel\`{a}-Ascoli-argument and a priori estimates to pass to a limit.

In Section \ref{sec shadow-flow}, we define the ``shadow'' at time $t$ to be the set $$\{x\in \R^n \colon |u(x,t)|<\infty\},$$ and interpret this as a weak solution to mean curvature flow of hypersurfaces in $\O$ with Dirichlet boundary condition on $\p \O$ using the notion of avoidance principle as defined in Section \ref{sec shadow-flow}. Roughly, the definition of weak solution implies that any classical solution starting inside the shadow stays inside and any classical solution starting outside stays outside.

An intersection of two smooth open sets is in general not smooth at the intersection of the boundaries.
In Section \ref{sec smoothing} we provide a lemma to locally mollify intersections of smooth open sets at the intersections of their boundaries while preserving certain curvature conditions.
\begin{theorem} \label{thm smoothing}
 Let $A,B\subset \R^n$ be open, $\p A,\p B \in C^\infty$ and suppose $A\cap B\neq \emptyset$ is bounded.
 Then for any $\eps>0$ there is an open set $\O$ with $\p \O \in C^\infty$ such that
 \[
  (A\cap B)\setminus(\p A\cap \p B)_\eps \subset \O \subset A\cap B,
 \]
 where $(\p A\cap \p B)_\eps:= \{x\in \R^n\colon \dist(x,\p A \cap \p B)<\eps\}$.

 Moreover, if the principal curvatures at every point of $\p A$ and $\p B$ lie in a symmetric, open or closed, convex cone $\Gamma\subset \R^{n-1}$ which contains the positive cone $\Gamma_+\subset \Gamma$, then $\O$ can be chosen such that the principal curvatures at any point of $\p \O$ lie in $\Gamma$, too. Here symmetric means invariant when interchanging $\kappa_i\leftrightarrow \kappa_j$.
\end{theorem}
\begin{example}
Choosing $\Gamma= \{ (\kappa_1,\ldots,\kappa_{n-1})\colon \kappa_i\geq 0 \text{ for all }i\}$ corresponds to convex subsets.
Choosing $\Gamma= \{ (\kappa_1,\ldots,\kappa_{n-1})\colon \sum \kappa_i\geq 0\}$ corresponds to mean-convex domains. This is what we use in Section \ref{sec existence}.
\end{example}

The problem of mollifying inside convex curvature cones was an open problem from the problem section of the conference {\it Geometric evolution equations} which took place in Konstanz in 2011. 
The proof uses distance functions to the boundaries and a mollified version of the minimum of these.
It can be read independently of the other sections and is applied most noteably in Section \ref{sec existence} in the approximation process.
The author expects that this result is of great use and will be widely applicable.

This article emerged from the author's master thesis and is meant to gather the main results. The author wishes to thank Oliver Schn\"urer for supervising the thesis and for his great support. The author also likes to thank Ben Lambert for helpful advice.

\section{Existence} \label{sec existence}

We prove the existence result Theorem \ref{thm existence} by approximating by auxiliary problems and using a priori estimates.

\begin{proof}[Proof of Theorem \ref{thm existence}]
We cut off the initial function $u_0$ in height by considering
\[ \ol{u}_{0,R}:= \widetilde{\max}\left(\widetilde{\min}(u_0,R),-R\right)\]
for $R>0$, where $\widetilde{\max}$ and $\widetilde{\min}$ are mollified versions of $\max$ and $\min$ respectively (defined analogously to \eqref{eq def of mollified min} setting $\delta=1/2$ there).
Next, we cut off the domain of definition $\O$ by intersecting with a ball $B_{2R}\equiv B_{2R}(0)$ and using Theorem \ref{thm smoothing}. This gives smooth open sets $\O\cap B_R \subset \O_R \subset \O\cap B_{2R}$ whose boundaries have nonnegative mean curvature $H[\p \O_R]\geq 0$.
Finally, we restrict the functions $\ol{u}_{0,R}$ to $\O_R$ and take a mollification to find smooth functions $u_{0,R}$ defined on $\O_R$ satisfying $\|u_{0,R}-\ol{u}_{0,R}\|_{L^\infty(\O_R)}<R^{-1}$ and $\|u_{0,R}-\ol{u}_{0,R}\|_{C^2(\p \O\cap B_R)}< R^{-1}$.

Now define $u_R$ as the solution of the auxiliary problem
\[
 \begin{cases}
  \dot{u}_R= \sqrt{1+|Du_R|^2}\, \divergence\left(\frac{Du_R}{\sqrt{1+|Du_R|^2}}\right)
  & \text{in } \O_R\times (0,\infty),\\
  u_R(x,t)= u_{0,R}(x)
  & \text{for } (x,t)\in \P(\O_R\times (0,\infty)).
 \end{cases}
\]
By \cite[Theorem 2.1]{Hui}, $u_R$ is well-defined and satisfies
\begin{align*}
  & u_R\in C^{0}(\ol{\O_R}\times [0,\infty))\cap C^\infty(\O_R \times (0,\infty))\\
 \text{and} \hspace{1cm} & Du_R\in C^0(\ol{\O_R}\times [0,\infty)).
\end{align*}
(By de Giorgi-Nash-Moser-estimates the spatial derivative $Du_R$ is even H\"older-continuous for some exponent.)

To be able to utilize the Arzel\`{a}-Ascoli-Theorem and to pass to a limit  and obtain a solution of the initial problem, we are going to need local a priori estimates.
These estimates will be local in space, time, and height: Due to the unboundedness in height we can only expect estimates at points $(x,t)$ depending on $|u_R(x,t)|$.

Since by Lemma \ref{lem gradient bounds} and Lemma \ref{lem sup bound} we have local gradient bounds at the boundary, local gradient estimates easily follow from the results of Section 2 in \cite{E+H}.
Using spheres as barriers we obtain H\"older-estimates in time with exponent $1/2$ (cf. Section 6 of \cite{MCF without Singularities}). This is sufficient to apply an Arzel\`{a}-Ascoli argument.

To use Arzel\`{a}-Ascoli for unbounded functions, simply compose with a homeomorphism $\Phi\colon [-\infty,\infty]\to [-1,1]$ which is smooth on $(-\infty,\infty)$. Then the gradient and H\"older-in-time estimates give locally uniform estimates for $\Phi\circ u_R$ where $\Phi\circ u_R\in (-1+\eps,1-\eps)$, which is sufficient: Locally in space-time there is for any $\eps>0$ a $\delta>0$ such that for almost all $R\in \N$ we have
\[
  |(x,t)-(y,s)|<\delta \Rightarrow |\Phi\circ u_R(x,t)-\Phi\circ u_R(y,s)|<\eps.
\]
By Arzel\`{a}-Ascoli a subsequence of $(\Phi \circ u_R)_{R\in \N}$ converges locally uniformly to a continuous function on $\ol{\O}\times [0,\infty)$ as $R\to \infty$. This correspondes to pointwise convergence of a subsequence of $(u_R)$ to a continuous function $u\colon \ol{\O}\times [0,\infty) \to [-\infty,\infty]$ and locally uniform convergence on $\ol{\O}\times [0,\infty) \cap \{|u|<\infty\}$.

Then, using the interior estimates for higher derivatives in \cite{E+H}, one has locally smooth convergence on $\O\times (0,\infty) \cap \{|u|<\infty\}$ and we see that $u$ solves the same differential equation there as the $u_R$, which completes the proof.
\end{proof}

Now we are going to establish the a priori estimates. Most importantly we need
\begin{lemma}[local gradient estimates at the boundary] \label{lem gradient bounds}
 Let $\O$ be as in the statement of Theorem \ref{thm existence}, $x_0\in \p \O$, $r,T>0$, $B_r := B_r(x_0)\cap \O$ and $\Gamma_r:= \p \O\cap B_r(x_0)$.
 Let $u\in C^{2;1}(B_{2r}\times (0,T))\cap C^0(\ol{B_{2r}}\times [0,T])$ be a solution of
 \[
  \begin{cases}
    \dot{u}= \sqrt{1+|Du|^2}\,\divergence\left( \frac{Du}{\sqrt{1+|Du|^2}} \right) & \mbox{in } B_{2r}\times (0,T),\\
    u(x,0)= u_0(x)& \mbox{for } x\in B_{2r},\\
    u(x,t)= \phi(x) & \mbox{for } x\in \Gamma_{2r},\, t>0.
  \end{cases}
 \]
 For the initial- and boundary data assume $u_0\in \operatorname{Lip}(\ol{B_{2r}})$ and $\phi\in C^2(\ol{B_{2r}})$. Assume $Du\in C^0(\ol{B_{2r}}\times [0,T])$.
 Then on $\Gamma_r\times [0,T]$ we have
 \[
    |Du|\leq C\left(n,\|u\|_{L^\infty(B_{2r}\times (0,T))}, \|u_0\|_{\operatorname{Lip}(B_{2r})}, \|\phi\|_{C^2(\ol{B_{2r}})},
    r, \Gamma_{3r}\right).
 \]
\end{lemma}
\begin{proof}
 A nonlocal version of this result can be found, for example, in \cite[Chapter 1.4]{Giusti}. There, one uses the barriers
 \[
  w^{\pm}:= \phi \pm \delta \log(1+\sigma d),
 \]
where $d$ is the signed distance function to $\p \O$ and the constants are chosen like $\delta \ll 1 \ll \sigma(\delta)$. This barrier works on an $\eps$-neighbourhood of $\p \O$, where $\eps$ is chosen as $\sigma^{-1/2}$ and so small, that we are in a tubular neighbourhood.
Because of $w^{\pm}=\phi$ on the boundary we find the desired gradient bound on the boundary by the comparison principle for Dirichlet boundary conditions.

To obtain a local version simply add $\pm (\|u\|_{L^\infty}+\|\phi\|_{L^\infty})\cdot \eta$ to the barriers, where $\eta\in C^\infty (\ol{B_{2r}})$ satisfies $0\leq \eta \leq 1$, $\eta\equiv 0$ on $B_r$ and $\eta\equiv 1$ on $B_{2r}\setminus B_{3r/2}$. The proof then works as in the nonlocal case but the choice of $\delta$ and $\sigma$ depends additionally on $r$ because the bounds on derivatives of $\eta$ depend on $r$.
\end{proof}
To apply Lemma \ref{lem gradient bounds} we need local $L^\infty$-estimates on $u$. For this purpose we have the following
\begin{lemma}[$L^\infty$-estimates near the boundary]\label{lem sup bound}
 Let $\O$ be as before, $R,T>0$ and $x_0$, $B_R$ and $\Gamma_R$ defined as in Lemma \ref{lem gradient bounds}.
 Let $u\in C^{2;1}(B_R\times (0,T))\cap C^0(\ol{B_R}\times [0,T])$ be a solution of \eqref{eq GMCF} on $B_R\times (0,T)$.
 Then there exists $r>0$ dependent on $B_R,\Gamma_R$ and $T$, such that
 \[
  \sup_{B_r\times (0,T)} |u| \leq C,
 \]
 where $C>0$ depends on the same quantities as $r$ and additionally depends on $\sup_{(B_R\times \{0\})\cup (\Gamma_R \times (0,T))}|u|$.
\end{lemma}
\begin{proof}
 Without loss of generality we may assume $x_0=0$ and that the inner normal to $\p \O$ at $x_0$ is $e_n=(0,\ldots,0,1)$.
 We are going to write the boundary locally as a graph:
 There is an open ball $B$ in $\R^{n-1}$ with center at the origin and $s>0$, such that
 \[
  B_R\cap (B\times (-s,s))= \{(\hat{x},x^n)\in B\times (-s,s)\colon h(\hat{x})<x^n\}
 \]
 and $\Gamma_R\cap (B\times (-s,s)) = \graph h$ for some function $h\in C^\infty(\ol{B})$.
 Because of $H[\p \O]\geq 0$, $h$ satisfies the differential inequality
 \[
  \divergence \left( \frac{Dh}{\sqrt{1+|Dh|^2}} \right) \geq 0.
 \]

 Take $0\leq \eta \in C^\infty_0(B)$, $\eta\neq 0$, such that $v_0:= h+\eta$ fulfils $|v_0|<s$.
 Finally, define $v\in C^\infty(B\times (-1,T+1))\cap C^0(\ol{B}\times [-1,T+1])$ to be the solution of
\[
  \begin{cases}
    \dot{v}=\sqrt{1+|Dv|^2}\divergence \left(\frac{Dv}{\sqrt{1+|Dv|^2}}\right)
    & \mbox{in } B \times (-1,T+1),\\
    v(\hat{x},t)=v_0(\hat{x})& \mbox{for } (\hat{x},t)\in \P(B\times (-1,T+1)).
  \end{cases}
\]
This way we have $v\in C^\infty(\ol{B}\times [0,T])$. By the maximum principle we also have $|v|<s$, and by the strong maximum principle $h(\hat{x})<v(\hat{x},t)$ holds for all $x\in B$ and $t>-1$.

Define $Q$ by
\[
 Q:= \{(\hat{x},x^n,t)\in B\times (-s,s)\times (0,T)\colon h(\hat{x})<x^n<v(\hat{x},t)\},
\]
and choose $r>0$ from our assertion such that $B_r\times (0,T)\subset Q$ and $B_r\times (0,T)$ has positive distance to $\graph v$.

Now we construct a barrier on $Q$ with the aid of the function $v$. On $Q$ define
\[ w(x,t):= \left(v(\hat{x},t)-x^n\right)^{-1}\equiv \left(v(x,t)-x^n\right)^{-1}, \]
setting $v(x,t)\equiv v(\hat{x},t)$ for $(x,t)\in Q$.
One easily verifies that the level-sets of $w$ move by mean curvature, so that $w$ solves
\[
 \dot{w}-\left(\delta^{ij}-\frac{w^i w^j}{|Dw|^2}\right) w_{ij} = 0.
\]
Using this fact we calculate
\begin{align*}
	&\dot{w}-\sqrt{1+|Dw|^2}\,\divergence\left(\frac{Dw}{\sqrt{1+|Dw|^2}}\right)=  \dot{w}-\left(\delta^{ij}-\frac{w^i w^j}{1+|Dw|^2}\right)w_{ij}\\	
	 &\hspace{2em} = -\frac{w^i w^j}{|Dw|^2\left(1+|Dw|^2\right)}w_{ij}
		\\
	 & \hspace{2em} = \frac{1}{1+|Dw|^2}
		\left(\frac{v_{ij}w^i w^j}{(v-x^n)^2|Dw|^2}
			-2|Dw|^2(v-x^n)\right)
		\\
	&\hspace{2em} \geq
		-\frac{\|D^2v\|_{L^\infty(Q)}}{|Dw|^2(v-x^n)^2}-2(v-x^n)
	\geq
		-4s^2\|D^2v\|_{L^\infty(Q)}-4s
		\\
	& \hspace{2em} \geq -c
\end{align*}
where $c>0$ is a constant that ultimately only depends on $B_R, \Gamma_R$ and $T$ through the above construction.
Finally,
\[
 w^{\pm} := \pm \left(w+ct+\sup_{(B_R\times \{0\})\cup(\Gamma_R\times (0,T))} |u|\right)
\]
are upper and lower barriers for $u$ on $Q$ respectively. Observe
\[w(y,t) \to \infty \mbox{ for } y\to x\in \P (Q)\setminus \left[ \big(B_R \times \{0\})\cup (\Gamma_R \times(0,T)\big)\right].\]
Now, on $B_r\times [0,T]$, $w^{\pm}$ and therefore $|u|$ are bounded by a constant as  in the assertion.
\end{proof}

\begin{remark}
 It is worth pointing out the solution is not unique in general:
 Consider two so called \emph{grim reaper curves} lying next to each other, i.\,e.\ the graph of the function
 \[
  u_0(x):= -\log |\sin x| \hspace{0.6cm} \text{for } 0\neq x\in (-\pi,\pi).
 \]
 We can write down a translating solution to \eqref{eq GMCF} with initial data $u_0$:
 \[
  \hat{u}(x,t):= t-\log |\sin x| \hspace{0.6cm} \text{for } 0\neq x\in (-\pi,\pi),\; t\in \R.
 \]
 But the solution $u$ we have constructed in the proof of Theorem \ref{thm existence} differs from this translating solution $\hat{u}$: The two grim reapers get connected at infinity.
 This is because we cut off the function $u_0$ at some height $R$ in the approximation process. To see that the approximating solutions $u_R$ do not converge to $\hat{u}$ we may consider the integral of the difference:
 \[
 \frac{d}{dt} \int\limits_{-\pi}^{\pi} \hat{u}(\cdot,t)-u_R(\cdot,t)
 = \int\limits_{\graph \hat{u}(\cdot,t)} \hat{K} \hspace{0.3cm}- \int\limits_{\graph u_R(\cdot,t)} K_R \,
 \geq 2\pi - \pi = \pi,
 \]
 where $\hat{K}, K_R$ are the respective curvatures. By the maximum principle the convergence $u_R \to u$ is monotone and thus $\int u_R(\cdot,t)\to \int u(\cdot,t)$ as $R \uparrow \infty$.
 We conclude $\int (\hat{u}-u)(\cdot,t) \geq \pi t$ and therefore $u\neq \hat{u}$.
\end{remark}

\section{The Shadow-flow} \label{sec shadow-flow}
In this section we are going to investigate the projections/shadows of graphical mean curvature flow, that is the sets $\{|u|<\infty\}$ where $u$ is as in Theorem \ref{thm existence}.
We show that this shadow is a weak solution of mean curvature flow, where we use the following notion of weak solution which is based on the avoidance principle.

For this section we do not need to assume constant Dirichlet boundary values.

Let $\O$ be as in the last section.
\begin{definition}[Weak solutions]
 A family $(A_t)_{t\in [0,\infty)}$ of open subsets of $\ol{\O}$ is called
 \begin{itemize}
  \item a supersolution to mean curvature flow if the following holds:
    For any family $(B_t)_{t\in [a,b]}$ of open sets, such that $B_t\Subset \O$ with $\p B_t\in C^\infty$, and such that $(\p B_t)_{t\in [a,b]}$ is a classical solution to mean curvature flow, we have
    \[
     \ol{B_a}\subset A_a \; \Rightarrow \; B_b\subset A_b.
    \]
  \item a subsolution to mean curvature flow, if $(\ol{\O}\setminus \ol{A_t})_{t\in [0,\infty)}$ is a supersolution.
  \item a weak solution of mean curvature flow, if $(A_t)_{t\in [0,\infty)}$ is both a super- and subsolution.
 \end{itemize}
 We shall call $(\p \O \cap A_t)_{t\in [0,\infty)}$ the boundary values of $(A_t)_{t\in [0,\infty)}$.
\end{definition}
\begin{remark}
 %\begin{itemize} \neueZeile
  %\item
  In the above setting in the definition of supersolution we even have $\ol{B_t}\subset A_t$ for all $t\in [a,b]$, when $(A_t)_t$ is a supersolution. To see that the closure is contained one can use, for instance, the translation invariance of classical flows.
  %\item

  If $(A_t)$ is a family of open subsets of $\ol{\O}$ such that $(\p A_t)$ is a classical solution to mean curvature flow, then, by the avoidance principle, $(A_t)$ is a weak solution.
  %\item 

  Let $x\in \p A_t \cap \O$ and suppose that $\p A_t$ is smooth in a ball $B_r(x)\subset \R^n$. If $(A_t)$ is a weak solution and $\p A_t$ is smooth in a spacetime-neighbourhood of $(x,t)$, then $(\p A_t)$ solves mean curvature flow at $(x,t)$ classically.
  \begin{proof}
   We may assume, that $\p A_t \cap B_r(x)$ is the graph of a smooth function. Using the result of the next section we find two smooth open sets, one lying inside $A_t \cap B_r$, the other inside $B_r(x)\setminus A_t$, and such that their boundaries coincide in a neighbourhood of $x$ with $\p A_t$. Taking $r$ to be small enough, small translations of these two open sets away from $\p A_t$ serve as barriers. This way it can be seen, that the normal velocity of $\p A_t$ at $(x,t)$ coincides with the mean curvature at that point.
  \end{proof}
  %\item

  Similar weak notions of mean curvature flow are the set-theoretic subsolutions of Ilmanen (\cite{Ilmanen}) or more generally the barriers of De Giorgi. Both of them were compared to the level-set flow (see \cite{Belletini Novaga} for a comparison of De Giorgi's barriers to  level-set flow). (See \cite{Topology Change} for a definition of set-theoretic subsolutions including boundary values.)
  %\item

  Note that our definition of weak solutions is not very useful where $\p A_t \subset \p \O$ (taking the boundary $\p A_t$ relative to $\ol{\O}$). This is because there is no space left for a classical solution, that could possibly push $\p A_t$ inwards into $\O$. To circumvent this, one could compare with classical solutions with boundary values that may not be written as the boundary of an open set and which can intersect $\p A_t$. But this would cause trouble in the methods we are going to use next. Another way would be to compare $\p A_t$ not only with classical solutions in $\O$ but with classical solutions that are boundaries of open subsets in $\R^n$ and which do not intersect the boundary values. This viewpoint has the disadvantage of not being intrinsically in $\ol{\O}$. The methods presented in the following also work if one adopts this definition for weak solutions.
% \end{itemize}
\end{remark}

\begin{proposition}
For any open set $A\subset \O$ there exists a weak solution of mean curvature flow $(A_t)_{t\in [0,\infty)}$ with $A_0= A$. Furthermore there is a smallest such weak solution $(A_t)_{t\in [0,\infty)}$: For any weak solution $(A'_t)_{t\in [0,\infty)}$ with $A\subset A'_0$ we have $A_t\subset A'_t$ for all $t\in [0,\infty)$.
\end{proposition}
\begin{proof}
Let $\mathcal{B}_0$ be the set of all families $(B_t)_{t\in [0,b]}$ of open sets such that $\ol{B_0}\subset A$ and $\p B_t\in C^\infty$ fulfils mean curvature flow. Define
\[
	A^{(0)}_t:= \bigcup \{B_t\colon (B_{t'})_{t'\in [0,b]}\in \mathcal{B}_0 \text{ with } t\in [0,b]\}.
\]
Then inductively define $\mathcal{B}_k$ and $A^{(k)}$ by setting $\mathcal{B}_k$ to be the set of all families $(B_t)_{t\in [a,b]}$ of open sets such that $\ol{B_a}\subset A^{(k-1)}_a$ and $\p B_t\in C^\infty$ fulfils mean curvature flow. Then set
\[
	A^{(k)}_t:= \bigcup \{B_t\colon (B_{t'})_{t'\in [a,b]}\in \mathcal{B}_k \text{ with } t\in [a,b]\}.
\]
Finally define the open sets $A_t:= \bigcup_{k\in \N} A^{(k)}_t$. Note that $A^{(k)}_t$ is a nondecreasing sequence of open sets. As a union of subsolutions $(A_t)$ is again a subsolution. To see that $(A_t)$ is a supersolution let $(B_t)_{t\in [a,b]}$ be a classical solution and let $\ol{B_a}\subset A_a$. Then by compactness $\ol{B_a}\subset A^{(k)}_a$ for some $k\in \N$. Therefore $\ol{B_b}\subset A^{(k+1)}_b\subset A_b$ and hence $(A_t)$ is a supersolution.

The second assertion is obvious from the construction.
\end{proof}

The following lemma concerning weak solutions will be useful. The result is trivial for classical solutions.
\begin{lemma} \label{lem weak cylinder}
Suppose $(A_t)_{t\in [0,\infty)}$ is a family of open subsets of $\ol{\O}$, such that $(A_t\times \R)_t$ is a weak solution of mean curvature flow in $\ol{\O}\times \R$.

Then $(A_t)_t$ is a weak solution of mean curvature flow in $\ol{\O}$.
\end{lemma}
\begin{proof}
We only show, that $(A_t)$ is a supersolution using the fact that $(A_t\times \R)$ is a supersolution.

Let $(B_t)_{t\in [a,b]}$ be a family of open sets such that $\p B_t\in C^\infty$ fulfils mean curvature flow, and $\ol{B_a}\subset A_a$. We need to show $B_b\subset A_b$.
In fact we prove $B_b\times \R \subset A_b\times \R$.
For this we approximate $B_a\times \R$ by bounded sets.

For $R>0$ take $\hat{K}^R$ to be a smoothed intersection of $B_a\times \R$ with $\{|x^{n+1}|<2R\}$ containing $B_a\times [-R,R]$ (use Theorem \ref{thm smoothing}). We may take the same closing ends for different $R>1$, to give curvature bounds on $\p \hat{K}^R$ independent of $R$. Then define $K^R:= \{x\in \hat{K}^R\colon \dist(x,\p \hat{K}^R)>R^{-1}\}$. Then we still have uniform curvature bounds for $R>R_0$ sufficiently large. Thus, by Proposition 4.1 of \cite{E+H} there are classical solutions $(M^R_t)_{t\in [a,a+\tau]}$ of mean curvature flow with $M^R_a=\p K^R$ for some $\tau >0$ independent of $R$. These solutions $(M^R_t)$ are written as graphs over the $\p \hat{K}^R$, which contain $\p B_a\times [-R,R]$. 

By interior estimates of \cite{E+H} and the uniqueness of the limit (see below) we find for $R\to \infty$ local convergence as graphs over $\p B_a\times \R$. This gives a solution of mean curvature flow which is written as a graph over $\p B_a\times \R$ and starts from there. Hence it coincides with $(\p B_t \times \R)_{t\in [a,a+\tau]}$.

(To see the uniqueness for smooth cylinders write a non-cylindrical solution as a graph over the initial cylinder. By the strong maximum principle the difference to the ordinary cylindrical solution attains no interior maximum. Then translate the non-cylindrical solution along the cylinder and again use interior estimates to find convergence to a new solution,  and do the translation in such a way that the difference of the limit to the cylindrical solution attains an interior maximum. This contradicts the strong maximum principle.)

Thus, we have shown that the corresponding flows $(K^R_t)_{t\in [a,a+\tau]}$ of the open sets starting from $K^R$ satisfy
\[
	\bigcup_{R>1} K^R_t = B_t \times \R, \hspace{0.5cm} \text{for } t\in [a,a+\tau].
\]

Let $(W_t)_{t\in [a,b]}$ be the smallest weak solution in $\R^{n+1}$ with $W_a= B_a\times \R$. The argument above has shown $W_t=B_t\times \R$ for $t\in [a,a+\tau]$. The same argument shows, that the maximal time-interval on which $(W_t)$ and $(B_t\times \R)$ coincide is open. It is easy to see that this maximal interval is also closed: Suppose $W_t=B_t\times \R$ for $t<t_0$. Since $(W_t)$ is the smallest solution $W_{t_0}\subset B_{t_0}\times \R$. To show the reverse inclusion let $x\in B_{t_0}\times \R$. Since $\bigcup_{t\in (a,b)} B_t\times \R \times \{t\}$ is open we find a ball-solution of mean curvature flow centred at $x$ that is contained in $(B_t\times \R)_t$ and contains $(x,t_0)$. This shows $x\in W_{t_0}$ since $W_t=B_t\times \R$ for $t<t_0$ and $W_t$ is a supersolution and hence contains the ball-solution.

Summarizing we find $W_t=B_t\times \R$ for all $t\in [a,b]$, i.\,e. the smallest weak solution starting from $B_a\times \R$ is $(B_t\times \R)_t$. As $A_t\times \R$ is a supersolution with $B_a\times \R\subset A_a\times \R$ we conclude $B_b\times \R \subset A_b\times \R$.
\end{proof}

A shadowflow is a weak solution:
\begin{theorem}
 Let $u\colon \ol{\O}\times [0,\infty)\to [-\infty,\infty]$ be continuous and smooth in the set $\{(x,t)\in \O\times (0,\infty)\colon |u(x,t)|<\infty\}$ and suppose $u$ satisfies \eqref{eq GMCF} in this set.
 Define $A_t$ to be the projection of $\graph u(\cdot,t) \cap \R^{n+1}$ onto $\ol{\O}$, that is $$A_t= \{x\in \ol{\O}\colon |u(\cdot,t)|<\infty\}.$$

 Then $(A_t)_{t\in [0,\infty)}$ is a weak solution of mean curvature flow with boundary values $(\{x\in \p \O \colon |u(x,t)|<\infty\})_{t\in [0,\infty)}$.
\end{theorem}
\begin{proof}
 By Lemma \ref{lem weak cylinder} it suffices to show that $A_t\times \R$ is a weak solution.

 First we prove that $(A_t\times\R)_t$ is a subsolution, which is slightly easier to see.
So let $(B_t)_{t\in [a,b]}$ be a family of open bounded subsets of $\R^{n+1}$ such that the boundaries $(\p B_t)_{t\in [a,b]}$ form a smooth solution of mean curvature flow.
Suppose $\ol{B_a}\subset (\O\setminus \ol{A_a})\times \R$. Assume for contradiction that there is $X\in B_b$ such that $X \notin (\O\setminus \ol{A_b})\times \R$. By openness of $B_b$ we may assume $X\in A_b\times \R$. Taking a vertical translation of $(B_t)$ we may further assume $X\in \graph u(\cdot,b)$. This contradicts the avoidance principle.

Now to see that $(A_t\times \R)_t$ is a supersolution, suppose $\ol{B_a}\subset A_a\times \R$ and assume for contradiction that there is $X\in B_b$ such that $X\notin A_b\times \R$.
W.\,l.\,o.\,g. we assume $u(X_1,\ldots,X_n,b)=+\infty$ and taking a vertical translation of $(B_t)$ we may assume $u(Y_1,\ldots,Y_n,a)<Y_{n+1}$ for all $Y\in \ol{B_a}$, i.\,e. $\ol{B_a}$ is above $\graph u(\cdot,a)$. But then by the avoidance principle $\ol{B_b}$ would be above $\graph u(\cdot,b)$ which leads to a contradiction.
\end{proof}

\section{Smoothing intersections while respecting curvature conditions} \label{sec smoothing}

We observe the following
\begin{lemma}\label{lem convex}
 Let $\Gamma\subset \R^{n}$ be an open or closed, symmetric and convex cone which contains the positive cone.
 Then the set of real symmetric $n\times n$-matrices with eigenvalues in $\Gamma$ is convex.
\end{lemma}
\begin{proof}
 Let $\Gamma$ be open. (The case of closed $\Gamma$ is handled similar.) Let $f\colon \R^{n}\to \R$ be the signed distance function to $\p \Gamma$ (which we assume to be nonempty) such that $\lambda\in \Gamma \iff f(\lambda)>0$. By the convexity of $\Gamma$, $f$ is concave. By the symmetry of $\Gamma$, $f$ is symmetric. And because $\Gamma$ contains the positive cone $f$ is increasing in each component of its argument. Then $F(A):= f(\lambda(A))$ is a concave function on the set of symmetric matrices, where $\lambda(A)$ denotes the eigenvalues of $A$ (see e.\,g. \cite[end of §3]{CNS}). Thus $\{A\colon F(A)>0\}$ is a convex subset of the symmetric matrices, which finishes the proof.
\end{proof}
\begin{proof}[Proof of Theorem \ref{thm smoothing}]
 The idea is to use distance functions and take a mollified version of $\min$ (denoted $\widetilde{\min}$) and to define $$\O:= \{\widetilde{\min}(\dist_{\p A},\dist_{\p B})>0\}$$ though we will not directly use the distance functions.

 {\bf 1.\ Altered distance functions and reference neighbourhoods.}
 First note that since $A\cap B$ is bounded it suffices to consider everything in a large ball. Let $d_A\in C^\infty(\R^n)$ be such that $d_A<0$ in $\R^n\setminus \ol{A}$ and in our large ball we have $d_A>0$ in $A$ and $d_A$ coincides with the signed distance function in a tubular neighbourhood of $\p A$.
 Let $g\in C^\infty(\R)$ be such that $g(0)=0,\,1\leq g'\leq 2$ and $g''(s)\leq -C$ for $|s|<\eps(C)$, where we choose $C>0$ later, and set $a:= g\circ d_A$.
 We derive
 \begin{align*}
  Da&= g'(d_A)\,Dd_A,\\
  D^2a&= g''(d_A)\,Dd_A \otimes Dd_A + g'(d_A)\,D^2d_A.
 \end{align*}
 In a tubular neighbourhood (in our large ball) $Da$ is an eigenvector of $D^2 a$ with eigenvalue $g''(d_A)$. The remaining eigenvalues are $g'(d_A)$ times the eigenvalues of $D^2 d_A$ which are
 \[
  \frac{-\kappa_i \circ \pi}{1-d_A\cdot \kappa_i\circ \pi} \hspace{0.5cm} (i=1,\ldots,n-1).
 \]
 Here $\kappa_i$ are the principal curvatures at the boundary and $\pi$ denotes the closest point projection onto the boundary.
 Note that inside $\ol{A}$ these eigenvalues of $D^2 d_A$ are not greater than the negated principle curvatures of the boundary. Because $\Gamma$ contains the positive cone, we find the negation of the eigenvalues of $D^2 a$ which correspond to eigenvectors orthogonal to $Da$ lie in $\Gamma$.

 Now choose $C>0$ from above such that the eigenvalue of $D^2 a$ which corresponds to the eigenvector $Da$ is the smallest (largest in absolute value) eigenvalue of $D^2 a$ in a neighbourhood of $\p A$ (still restricted to a large ball). We refer to this neighbourhood restricted to $\ol{A}$ as the reference neighbourhood of $\p A$ (in our large ball). In the reference neighbourhood the negations of the eigenvalues of $D^2 a|_V$ are in $\Gamma$ for any $(n-1)$-dimensional hyperplane $V$.

 Analogously we define $b$ with respect to $\p B$ and the reference neighbourhood of $\p B$.

 {\bf 2.\ Construction.}
 Let $f\in C^\infty(\R)$ be a function with the following properties
 \begin{enumerate}[(i)]
  \item $\min\{s,0\}-1 < f(s)<\min\{s,0\}$ for $|s|<1$,  
  \item $f(s)= \min \{s,0\}$ for $|s|\geq 1$,
  \item $0\leq f' \leq 1$,
  \item $f''\leq 0$.
 \end{enumerate}
 Define
 \begin{equation} \label{eq def of mollified min}
  \begin{split}
  \Phi \colon (0,1) \times \R^n &\to \R, \\
  (\delta,x) & \mapsto
  \delta f \left(\frac{a(x)-b(x)}{\delta}\right)+b(x).
  \end{split}
 \end{equation}
 This is a mollified version of $\min(a,b)$ with parameter $\delta$. In fact
 \begin{equation}
  \min\{a(x),b(x)\}-\delta < \Phi(\delta,x)\leq \min\{a(x),b(x)\}
  \label{eq Phi and min}
 \end{equation}
 for all $x\in \R^n$ and $\delta \in (0,1)$. We will choose $\O$ from the assertion of the form
 $$\O_\delta :=\{x\in \R^n \colon \Phi(\delta,x)>0\}$$
 for an appropriate choice of $\delta$.

 {\bf 3.\ Inclusions}. Because of $A\cap B = \{\min(a,b)>0\}$ it is obvious from \eqref{eq Phi and min} that $\O_\delta \subset A\cap B$ holds for all $\delta\in (0,1)$. To check the other inclusion let $x\in A\cap B\setminus (\p A\cap \p B)_\eps$.
 By continuity there is $0<\delta_1=\delta_1(\eps)$ independent of $x$, such that
 \begin{equation}\label{eq delta_1}
  \max\{a(x),b(x)\}>\delta_1.
 \end{equation}
 Now choose $\delta \leq \delta_1/2$ and distinguish two cases: Suppose $|a(x)-b(x)|>\delta$. Then by property (ii) of $f$ we find
 \[
  \Phi(\delta,x)= \min\{a(x),b(x)\}>0.
 \]
 If on the other hand $|a(x)-b(x)|\leq \delta$ then by \eqref{eq Phi and min}, \eqref{eq delta_1}, and $\delta\leq \delta_1/2$
 \[
  \Phi(\delta,x)> \min\{a(x),b(x)\}-\delta \geq \delta_1-2\delta \geq 0.
 \]
 In summary we find the claimed inclusions, provided that $\delta$ is sufficiently small.

 {\bf 4.\ Smoothness.}
 We compute
 \[
  D\Phi= \left (f\left (\frac{a-b}{\delta}\right )
  -f'\left (\frac{a-b}{\delta}\right )\frac{a-b}{\delta},
  \, f'\left (\frac{a-b}{\delta}\right )(Da-Db)+Db \right ).
 \]
 Assuming $\Phi=0$ implies
 \[ f\left (\frac{a-b}{\delta}\right )=-\frac{b}{\delta},\]
 and therefore
 \[
  \p_\delta \Phi= 0
  \iff
  b+\underbrace{f'\left (\frac{a-b}{\delta}\right )}_{\in [0,1]}(a-b)
  =0.
 \]
 This occurs only if one of the following is true (note that $\Phi=0$ already implies $a,b\geq 0$, by \eqref{eq Phi and min})
 \begin{enumerate}[(I)]
  \item $a=0$ and $b=0$
  \item $a=0$ and $f'(\tfrac{a-b}{\delta})=1$
  \item $b=0$ and $f'(\tfrac{a-b}{\delta})=0$.
 \end{enumerate}
 (I) implies $\Phi\neq 0$, a contradiction. In case (II) we find $\p_x \Phi= Da \neq 0$, because $Da \neq 0$ on $\p A$, which is where $a=0$ holds.
 Case (III) is treated analogously to case (II).

 Summarizing we obtain $D \Phi \neq 0$ where $\Phi=0$.
 The implicit function theorem shows, that $\Phi^{-1}(0)$ is a smooth $n$-dimensional submanifold of $(0,1)\times \R^n$ with normal $\frac{D\Phi}{|D\Phi|}$.

 By applying Sard's theorem to the mapping $\Phi^{-1}(0)\to (0,1), \, (\delta,x)\mapsto \delta$ one can show, that for almost all $\delta\in (0,1)$, $\p_x \Phi(\delta,\cdot)\neq 0$ where $\Phi(\delta,\cdot)=0$, so that by the implicit function theorem again $\O_\delta$ is smooth for almost all $\delta \in (0,1)$ and $\p \O_\delta = \{x\in \R^n\colon \Phi(\delta,x)=0\}.$

 We choose $\O:= \O_{\delta_0}$ for such an $\delta_0\in (0,1)$ sufficiently small, that the asserted inclusions hold.

 {\bf 5.\ Curvature condition.}
 Let $x_0\in \p \O$ and assume without loss of generality $x_0=0$ and further assume that the tangent space of $\p \O$ at $x_0=0$ is orthogonal to $e_n$. We identify the tangent space and $\R^{n-1}\equiv \{(\hat{x},x^n)\in \R^n\colon x^n=0\}$. Locally $\p \O$ is a graph over the tangent space at $x_0$: Let $w\in C^\infty(\R^{n-1})$ and $r>0$ with
\[
  B_r (0)\cap \O = B_r (0)\cap
  \{(\hat{x},x^n)\in \R^n \colon w(\hat{x})<x^n\}.
\]

We write $\Phi_{\delta_0}\equiv \Phi(\delta_0,\cdot)$ and $\hat{D}$ for differentiation with respect to the first $n-1$ components.

The following holds in a neighbourhood of $0\in \R^{n-1}$:
\begin{align}
 \Phi_{\delta_0}(\hat{x},w(\hat{x}))&=0, \label{eq Phi=0}\\
 \hat{D}w(0)&=0, \label{eq Dw=0}\\
 \p_n\Phi_{\delta_0}(0)&= |D\Phi_{\delta_0}(0)|. \label{eq p_n Phi}
\end{align}

Differentiating \eqref{eq Phi=0} twice and using \eqref{eq Dw=0} and \eqref{eq p_n Phi} we obtain
\[
 \hat{D}^2 w(0)
  =-\frac{\hat{D}^2\Phi_{\delta_0}(0)}{|D\Phi_{\delta_0}(0)|} 
\]
Note that the eigenvalues of $\hat{D}^2 w(0)$ are the principle curvatures of $\p \O$ at $x_0$. Hence, it remains to prove that the negated eigenvalues of $\hat{D}^2 \Phi_{\delta_0}(0)$ lie in $\Gamma$.

We compute
\begin{equation} \label{eq 2nd derivative Phi}
\begin{split}
 \hat{D}^2 \Phi_{\delta_0}
  &= \frac{1}{\delta_0}f''\left(\frac{a-b}{\delta_0}\right)
    (\hat{D}a-\hat{D}b)\otimes(\hat{D}a-\hat{D}b)\\
  &\quad  + f'\left(\frac{a-b}
  {\delta_0}\right)\left(\hat{D}^2a-\hat{D}^2b\right)+\hat{D}^2b.
  \end{split}
\end{equation}
 The matrix $(\hat{D}a-\hat{D}b)\otimes(\hat{D}a-\hat{D}b)$ is positive semi-definite and $f''\leq 0$ (property (iv)). Since $\Gamma$ contains the positive cone it suffices to consider the second term in \eqref{eq 2nd derivative Phi}. We distinguish three cases, below. But first we observe the following: We may assume that $\eps>0$ is so small that $(\p A\cap \p B)_\eps \cap (A\cap B)$ is contained in the intersection of the reference neighbourhoods of $\p A$ and $\p B$. Then $x_0=0\in \p \O$ is in $\p A$ or $\p B$ or in the reference neighbourhoods of both $\p A$ and $\p B$.
\begin{enumerate}[(I)]
 \item
  $a(0)-b(0)\geq \delta_0$: Then
  \[
   f'\left(\frac{a(0)-b(0)}{\delta_0}\right)=0.
  \]
  The term in question becomes $\hat{D}^2 b$. Moreover, in this case $0=\Phi_{\delta_0}(0)=b(0)$ holds. That is why we are in the reference neighbourhood of $\p B$, where the negations of the eigenvalues of $\hat{D}^2 b$ lie in $\Gamma$.
 \item
  $a(0)-b(0)\leq -\delta_0$: This case is treated similiar.
\item
  $|a(0)-b(0)|<\delta_0$: Here, by property (i) of $f$, $0\notin \p A,\p B$. Therefore $0$ must be in the reference neighbourhoods of $\p A$ and $\p B$.
  As a convex combination of two matrices, whose negated eigenvalues lie in $\Gamma$, the negated eigenvalues of
  \[
  f'\left(\frac{a-b}
  {\delta_0}\right)\left(\hat{D}^2a-\hat{D}^2b\right)+\hat{D}^2b,
  \]
the matrix in question, are in $\Gamma$ by Lemma \ref{lem convex}.
\end{enumerate}
In any case, the principal curvatures of $\p \O$ at the point $x_0=0$ lie in $\Gamma$.
\end{proof}

%-----------------------   bibliography   ----------------------------------------

\bibliographystyle{amsplain}

\end{document}